\documentclass[11pt]{article}
\usepackage{amsmath,amsthm,amssymb}
\usepackage[margin=1in]{geometry}
\usepackage{url}
\usepackage{tikz}
\usepackage{float}
\newcommand{\polya}{P\'{o}lya}
\newcommand\vecX{{\bf X}}
\newcommand\matA{{\bf A}}
\newcommand\matB{{\bf B}}
\newcommand{\Exp}{{\rm Exp}}
\newcommand{\E}{\mathbb{E}}
\newcommand{\convD}{\, \overset{D}{\longrightarrow} \,}
\newcommand{\V}{\mathbb{V}{\rm ar}}
\newcommand{\field}{\mathcal{F}}

\newcommand{\Given}{\, \Bigg{\vert} \,}

\newtheorem{theorem}{Theorem}
\newtheorem{cor}{Corollary}
\newtheorem{lem}{Lemma}

\newtheorem{remark}{Remark}

\begin{document}  
\begin{center}
{\Large \bf Distributions in the constant-differentials P\'{o}lya process}

\bigskip
{\large \bf Hosam M.\ Mahmoud\footnote{Department of Statistics, The George Washington University, Washington, D.C. 20052, U.S.A.} \qquad and \qquad
Panpan Zhang\footnote{Department of Statistics, University of Connecticut, Storrs, CT 06269, U.S.A.}}

\bigskip

\today
\end{center}

\bigskip\noindent
{\bf Abstract:}
We study a class of unbalanced constant-differentials \polya\ processes on white and blue balls. We show that the number of white balls, 
the number of blue balls, and the total number of balls,
when appropriately scaled, all converge in distribution to %a
%Panpan: Remove the article "a," as "random variable" is plural in the sentence.
gamma random variables with parameters depending on the differential index and the amount of ball addition at the epochs, but
not on the initial conditions. 
%Panpan: What is better? Initial conditions or initial numbers? What do you prefer?
The result is obtained by an analytic approach utilizing partial differential equations.


	\section{Introduction}
{\em \polya\ urn} has been a popular model in the research of applied probability and stochastic analysis due to its simplicity and versatility. In general, the dynamics of \polya\ urns are considered on a discrete-time scale. Two-color \polya\ urn schemes have found applications in many fields. Two of the most classical \polya\ urn models are the {\em \polya-Eggenberger urn} (see~\cite{Polya}) and the {\em Ehrenfest urn} (see~\cite{Ehrenfest}), which were respectively used to model disease contagion and gas diffusion. In modern times, different types of \polya\ urns are utilized to study scientific models from more diverse areas, such as the {\em play-the-winner scheme} in clinical trials (e.g.,~\cite{Rosenberger, Wei}), and recursive tree models and searching algorithms in computer science (e.g.,~\cite{Mahmoud2003}). In addition, mathematicians, probabilists and theorists are committed to developing rigorous methods for characterizing the dynamics of \polya\ urns analytically. Some exactly solvable \polya\ urn models are summarized and listed in~\cite{Flajolet}, and the extensions to solvable randomized \polya\ urns appear in~\cite{Morcrette}. We refer the interested readers to~\cite{Bala} for a general background, and to~\cite{Johnson, Mahmoud} for textbook style expositions.

Embedding discrete-time processes into continuous time has been around for a long time; see~\cite{Kac}, for example. The idea was first employed in the context of urns in~\cite{Athreya}. A Poisson-like transformation was introduced to create a bridge between discrete-time urn schemes and the associated continuous-time urn processes obtained by embedding. 
In some articles or texts, this type of transformation is called ``poissonization.'' In the literature, sometimes the poissonized \polya\ urns are called {\em \polya\ processes}, first 
so named in~\cite{Sparks}. An inverse transformation (called ``depoissonization'') 
was also reported in~\cite{Athreya}. Problems arising in depoissonization are highly nontrivial, and the 
inverse transformation itself seems elusive till today.

\section{The \polya\ process}

To begin with, we give a few words about \polya\ urns. A two-color \polya\ urn scheme is an urn 
containing balls of up to two  
colors, say white and blue, and evolving over time according to some predetermined rules, which govern the dynamics of urn evolution. At each time point, a ball is drawn at random (all balls in the urn being equally likely),
its color is observed, and it is placed back in the urn along with some new balls according to the pertinent addition rule. This is a sampling process with replacement and possible growth (or shrinkage) of the urn population. The addition rules are as follows. When a white ball is drawn, it is returned to the urn together with $a$ white balls and $b$ blue balls; when a blue ball is drawn, it is returned to the urn together with $c$ white balls and $d$ blue balls. In general, these numbers can be positive, zero or negative, where a negative value is interpreted as removing that many balls out of the urn, or even random. 
The class of urn processes we investigate in this manuscript calls only for fixed replacements.

An urn is called {\em tenable} if the rule can be executed forever on all possible stochastic paths, and it never gets ``stuck,'' due to being empty (in which case the next ball cannot be drawn) or not having enough balls to be removed for at least one color, while executing the ball replacement rules after a draw.

It is customary to represent these dynamics by the (ball addition) {\em replacement matrix}
\begin{equation*}
\begin{pmatrix}
a & b
\\ c & d
\end{pmatrix};
\end{equation*}
the rows of this matrix are indexed with white and blue, respectively from top to bottom, and the columns are indexed with white and blue, respectively from left to right.

An associated \polya\ urn process is obtained by embedding a discrete-time \polya\ urn scheme into continuous time. It can be thought of as a renewal process, where the predetermined rule is executed at each renewal point (called {\em epoch}). It was suggested in~\cite{Athreya} to use exponential random variables with mean $1$ (which we call $\Exp{(1)}$) for modeling {\em interarrival time}, as the exponential distribution has tractable properties. In particular, the exponential distribution appears to be a useful device for the embedding owing to its appealing {\em memoryless} property and scalability. 

More precisely, consider that each ball in the 
urn is endowed with a clock that rings in $\Exp{(1)}$ time, and all these clocks are independent of each other and of any other random variables related to the past. At an epoch, the pertinent rules associated with its replacement matrix are executed. All newly-added balls are endowed with their own independent clocks. The execution of the rules takes place instantaneously after a draw, without any time lapse.  

Most investigations of urn schemes or processes with a constant ball replacement matrix assume a condition called {\em balance}, in view of which the row sum of the ball replacement matrix is constant, and thus the total number of balls added at each step remains unchanged; see~\cite{Chauvin2011, Chauvin2015} for instance. The balance condition leads to a mathematically-convenient model, whereby conditional proportions have a {\em deterministic} denominator. 

In this short communication, we deviate from the balance condition, a situation in which the total number of balls is a nondegenerate {\em  random variable}. We are interested in a class of two-color \polya\ urns with the replacement matrix
\begin{equation}\matA = 
\begin{pmatrix}
-a & -a
\\ a & a
\end{pmatrix},
\label{Eq:constant}
\end{equation}
for integer $a\ge 1$. We characterize the asymptotic behavior of the possonized urn. In this particular class of urn models, we always add or subtract balls of the two colors in the same amount. The difference between the number of blue and number of white balls remains the same at all times. We call $\Delta$ the {\em differential index} of the urn, as it appears as
a governing parameter in the underlying asymptotic distributions. 

Note that we require $\Delta \ge 1$ for tenability; if $\Delta \le 0$, the initial number of white balls exceeds the number of blue balls, and it is possible for the process to get locked in a path that depletes the blue balls and ultimately the process halts. Besides, this class of \polya\ processes does not belong to generalized urn processes investigated in the seminal article~\cite{Janson}, as several assumptions therein are not satisfied, rendering our research novel. The study of this class in discrete time presents many algebraic difficulties owing to the lack of balance condition. It was pointed out by Basil Morcrette during his presentation in the 24th International Meeting on the Analysis of Algorithms (AofA 2013) that the problems of unbalanced \polya\ urn are challenging even by using the very powerful analytic combinatoric methods, since the ordinary differential equations generated by the so-called ``analytic urn'' technique are significantly harder to solve than those that appear upon applying the method to a balanced urn. 

\section{Continuous-time constant-differentials \polya\ processes}
\label{Sec:continuous}
In this section, we investigate the constant-differentials \polya\ processes that are generated by the \polya\ urns associated with the replacement matrix in~(\ref{Eq:constant}) 
%Panpan: add "in" after replacement matrix
and a positive differential index. 
Our strategy is to exploit the methodology developed in~\cite{balajiM}, and our goal is to characterize the asymptotic distribution of $W(t)$ and $B(t)$, the number of white balls and blue balls, respectively, in the urn at time $t \in \mathbb{R}_0^{+}$, after proper scaling. The principal idea is to establish a partial differential equation (PDE) that governs the dynamics of the \polya\ process as time goes by. 

Let $\phi(t, u, v):= \E\left[e^{uW(t) + vB(t)}\right]$ be the joint moment generating function of $W(t)$ and $B(t)$. The PDE in~\cite{balajiM} is valid
%Panpan: add "valid" before for all kinds of tenable.
for all kinds of tenable urn processes 
%Panpan: add "urn processes" after tenable
(not necessarily balanced).
\begin{lem}[\mbox{\cite[Lemma 2.1]{Balaji}}]
	%Panpan: Specify the exact lemma/theorem of the PDE in Balaji et al. 2006.
	\begin{equation}
	\frac{\partial \phi(t, u, v)}{\partial t} + \left(1 - e^{au + bv}\right) \frac{\partial \phi(t, u, v)}{\partial u} + \left(1 - e^{cu + dv}\right) \frac{\partial \phi(t, u, v)}{\partial v} = 0.
	\label{Eq:generalPDE}
	\end{equation}
\end{lem}
The general solution to PDE~(\ref{Eq:generalPDE}) is known, but not in a closed form for most cases. In essence, the solution is an integration along characteristic curves. These integrals are often too difficult to obtain.

There are only a few cases for which the solutions are developed: the forward and backward diagonal processes in~\cite{balajiM}, the Ehrenfest processes in~\cite{Balaji}, a class of zero-balanced processes with replacement matrix of {\rm Bernoulli} entries in~\cite{Sparks}, the Apollonian processes in~\cite{Zhang}, the triangular urn processes in~\cite{Chen} and the Bagchi-Pal processes in~\cite{Chen2018}. 

In particular, the PDE for constant-differentials \polya\ processes is
\begin{equation}
\frac{\partial \phi(t, u, v)}{\partial t} + \left(1 - e^{-au - av}\right) \frac{\partial \phi(t, u, v)}{\partial u} + \left(1 - e^{au + av}\right) \frac{\partial \phi(t, u, v)}{\partial v} = 0.
\label{Eq:cdPDE}
\end{equation}
We are able to simplify PDE~(\ref{Eq:cdPDE}) by observing that the (marginal) process of $B(t)$ is fully specified by the (marginal) process of $W(t)$ since the difference $\Delta = \Delta(t) = B(t) - W(t)$ remains constant for all $t$. We thus only need to focus on the evolutionary behavior of $W(t)$, i.e., the moment generating function %on 
%Panpan: change "on" to "of"
of $W(t)$ only. Consider $\psi(t, u) := \E\left[e^{u W(t)}\right] = \phi(t, u, v = 0)$. 
%Panpan: I suggest to define $\psi(t, u)$ before the lemma, as it appears in Lemma 3.2 before the definition is given (in the proof).
\begin{lem}	\begin{equation}
	\label{Eq:reducedPDE}
	\frac{\partial \psi(t, u)}{\partial t} + \bigl(2 - e^{-au} - e^{au}\bigr) \frac{\partial \psi(t, u)}{\partial u} + \Delta \, \bigl(1 - e^{au}\bigr)  \psi(t, u) = 0.
	\end{equation}
\end{lem}
\begin{proof}
	%Let $\psi(t, u) := \E[e^{u W(t)}] = \phi(t, u, v = 0)$. 
	Noticing
	\begin{align*}
	\frac{\partial \phi(t, u, v)}{\partial v} {\Given}_{v = 0} &= \E\left[B(t) e^{uW(t)}\right] 
	\\ &= \E\Bigl[\bigl(W(t) + \Delta\bigr) e^{uW(t)}\Bigr] 
	\\ &= \frac{\partial \psi(t, u)}{\partial u} + \Delta \, \psi(t, u),
	\end{align*}
	%Panpan: shall we remove $\E\bigl[B(t) e^{uW(t)}\bigr]$ in the above so as to write the display in one line?
	we are able to rewrite PDE~(\ref{Eq:cdPDE}) in terms of $\psi(t, u)$; that is,
	\begin{equation*}
	\frac{\partial \psi(t, u)}{\partial t} + \left(1 - e^{-au}\right) \frac{\partial \psi(t, u)}{\partial u} + \left(1 - e^{au}\right) \Bigl(\frac{\partial \psi(t, u)}{\partial u} + \Delta \, \psi(t, u)\Bigr)  = 0,
	\end{equation*}
	which is equivalent to the stated PDE.
\end{proof}

%The 
%Panpan: remove "The." 
PDE~(\ref{Eq:reducedPDE}) is an initial-value problem, to be solved under the boundary condition $\psi(0, u) = e^{u W(0)}$. The problem is amenable to the {\em method of characteristics} in~\cite{Levine}. 
%Panpan: make "method of characteristics" italic.
In fact,
the PDE is of a known type called the {\em transport equation}. The presence of the term $\Delta \, \left(1 - e^{au}\right)  \psi(t, u)$ renders it inhomogeneous. 
\begin{theorem}
	For time $t \in \mathbb{R}^{+}$,
	%Panpan: add the domain of t.
	let $W(t)$ and $B(t)$ be respectively the number of white and blue balls in a constant-differentials \polya\ process on white and blue balls with the replacement matrix 
	$$\begin{pmatrix}
	-a & -a
	\\ a & a
	\end{pmatrix},$$
	and of differential index  $\Delta =B(0) - W(0) \ge 1$. As $t\to\infty$, we have
	$$\frac{W(t)}{t} \convD {\rm Gamma}\Bigl(\frac \Delta a, a^2\Big).$$
\end{theorem}
\begin{proof}
	The characteristic curves for PDE~(\ref{Eq:reducedPDE})
	%Panpan: complete the description of the characteristic curves. I feel that this makes it clearer.
	are the solutions to the {\em ordinary differential equations} (ODEs):
	$$C'(s) = 2 - e^{-as} - e^{as}, \qquad C(t) = u.$$
	The characteristic curves are
	$$C(s) = \frac 1 a \ln \left(\frac {a(s -t) e^{au} - a(s-t) + e^{au} }
	{a(s -t) e^{au} - a(s-t) + 1} \right).$$ 
	We need the intersection point with the $t$-axis to construct the solution; it is	
	$$C(0) = \frac 1 a \ln \left(\frac {at e^{au} - at - e^{au} } {at e^{au} - at - 1} \right).$$
	Unlike the case of homogenious transport equation, where the PDE solution is constant over the characteristic curves, in an  inhomogenious case like the one at hand, the PDE solution varies on the characteristic curves.
	In our case, the PDE solution must satisfy the ODE
	$$v'(s) + \Delta(1 -e^{aC(s)}) v(s) = 0.$$
	An initial condition on the PDE is $\psi(0, t) 
	= \E[e^{W(0) u}] = e^{W(0) u}$. 
	It is required then to solve the ODE for $v(s)$ under the initial condition $v(0) = e^{W(0) C(0)}$. The ODE in question is first-order and linear, with a standard solution. As a solution to the PDE, we obtain
	$$\psi(t, u) = \frac{1}{(1 + at - at e^{au})^{\Delta / a}} \left(\frac{at e^{au} - at - e^{au}}{at e^{au} - at - 1}\right)^{W(0) / a}.$$
	
	Next, we derive the asymptotic distribution of $W(t)$ by considering scale $t$, i.e., $W(t)/t$, as $t \to \infty$. Recall the moment generating function $\psi(t, u)$. We replace the dummy variable $u$ by $x/t$, where $x$ is the new dummy variable, and obtain
	$$\E[e^{x W(t)/t}] = \frac{1}{(1 + at - at e^{ax/t})^{\Delta / a}} \left(\frac{at e^{ax/t} - at - e^{ax/t}}{at e^{ax/t} - at - 1}\right)^{W(0) / a}.$$
	Consider the local expansion of exponential function,
	$$e^{ax/t} = 1 + \frac{ax}{t} + O\left(\frac{1}{t^2}\right).$$
	We come up with
	\begin{align*}
	\E[e^{x W(t)/t}] &= \frac{1}{(1 - a^2x + O(1/t))^{\Delta / a}} \left(\frac{a^2x - 1 + O(1/t)}{a^2x - 1 + O(1/t)}\right)^{W(0) / a}
	\\ &\to (1 - a^2x)^{-\Delta/a}.
	\end{align*}
	In the last convergence relation, the right-hand side is
	the moment generating function of a gamma random variable  
	with shape parameter $\Delta/a$ and scale parameter $a^2$. We thus conclude the stated convergence in distribution
	according to {\em L\'{e}vy's continuity theorem}; see~\cite[p.\ 172]{Karr}.
\end{proof}
\begin{cor}
	\label{Cor:meanvar}
	The mean and variance
	of $W(t)$ are given by 
	\begin{align*}
	\E\left[W(t)\right] &= W(0) + a\Delta t  ,\\
	\V\left[W(t)\right] &= a^2 t(W(0) +  \Delta + a\Delta t).
	\end{align*}
	%Consequently, we obtain the variance of $W(t)$ as follows:
	%Panpan: remove the sentence above, as it seems to make no sense here
\end{cor}
\begin{proof}
	The mean, at time $t$, is readily obtained by taking the derivative of $\psi(t,u)$ at $u = 0$. Likewise, the second moment of $W(t)$ is obtained by taking the derivative of $\psi(t,u)$ twice at $u = 0$, %obtaining
	%Panpan: too many "obtains"
	i.e.,
	$$\E\left[W^2(t)\right] = W^2(0) + (2 a^2 t + 2at\Delta)W(0) + a^2 t \Delta(at + t \Delta + 1).$$
	The variance then follows after simplifying $\E\left[W^2(t)\right]- \E^2\left[W(t)\right]$.
\end{proof}

\begin{cor}
	\begin{align*}
	%\frac{W(t)}{t} &\almostsure  {\rm Gamma}\Bigl(\frac \Delta a, a^2\Bigr),\\
	%Panpan: the above expression has been stated in the main theorem
	%Panpan: change all "a.s." convergence to convergence in distribution
	\frac{B(t)}{t} &\convD  {\rm Gamma}\left(\frac \Delta a, a^2\right),\\
	\frac{\tau(t)}{t} &\convD 2 \, {\rm Gamma}\left(\frac \Delta a, a^2\right).
	%Panpan: add a single space above
	\end{align*}
\end{cor}
\begin{proof}
	In view of the constant differential condition, we have
	$$\tau(t) = W(t) + B(t) = W(t) + \bigl(W(t) - \Delta\bigr) = 2W(t) - \Delta.$$
	With $W(t) /t$ convergent in distribution and $\Delta /t \to 0$, by {\em Slutsky's theorem} (see~\cite[p.\ 146]{Karr}) we can assert the corollary.
\end{proof}

\section{An alternative approach via martingale}
%Panpan: rewrite quite a bit over here:
%To strengthen the distributional result into almost sure convergence, we formulate a martingale to be poised to use the martingale convergence theorem. 

We use the superscript $\top$ on a matrix to denote its transpose, and consider $\matB = \matA^{\top}$.
It is folklore that $e^{-t \matB}\vecX(t)$, where $\vecX(t)$ is a random vector containing the number of different kinds of balls in a \polya\ process at time~$t$, is a continuous-time martingale, which has been observed in many classes of \polya\ processes. This is shown in~\cite{Janson} for a broad class of \polya\ processes that have a particular eigenvalue structure, with a positive real principal eigenvalue. However, this condition is not satisfied in the case of constant-differentials \polya\ processes, where the principal eigenvalue is a repeated $0$. We prove next that indeed,  $e^{-t \matB}\vecX(t)$ is a martingale, too, for the class of constant-differentials \polya\ processes.
\begin{lem}
	The vector $e^{-t \matB}\vecX(t)$ is a bivariate martingale with respect to the natural filtration generated $\field_t$ by the evolution of a constant-differentials \polya\ process associated with Matrix~(\ref{Eq:constant}).
	%an urn with a constant differential index. 
	%Panpan: rewrite the above sentence a little bit
\end{lem}
\begin{proof}
	Let $\vecX(t)$ be the vector $\bigl(W(t), B(t)\bigr)^{\top}$.
	Note that $\matB$ is a {\em nilpotent matrix}, 
	%Panpan: do we need a reference for nilpotent matix, or do you think this is something quite well known to the math/stat community?
	as $\matB^2 = \bf 0$. 
	In view of this nilpotency, the matrix %$e^{t\matB}$ 
	%Panpan: change $e^{t\matB}$ to $e^{-t\matB}$
	$e^{-t\matB}$ is easily found (from the series expansion) to be ${\bf I} - t\matB = \left[\begin{smallmatrix} 1+at &-at \\ at &1-at\end{smallmatrix}\right]$. We thus have
	
	\begin{align*}
	\E\left[e^{-t \matB}\vecX(t) \, \big | \, \field_s\right] 
	&= \E\left[\begin{pmatrix} W(t) + atW(t) - atB(t) \\   
	atW(t) + B(t) - atB(t)
	\end{pmatrix} \Given \field_s\right]  \\
	& = \E\left[\begin{pmatrix} W(t) -a\Delta t \\   B(t) -a\Delta t
	\end{pmatrix} \Given \field_s\right].  
	\end{align*}		             
	Let us interpret the top row, which reads
	$$\E\left[W(t) -a\Delta t  \, \big | \, \field_s\right] = \E\left[W(t)  \, \big | \, \field_s\right] - a\Delta t.$$
	We consider $\E\left[W(t) -a\Delta t  \, \big | \, \field_s\right]$ as the status of %the
	%Panpan: remove "the"
	white balls, in an urn process %started 
	%Panpan: change started to starts
	starts at time $s$ with an initial number $W(s)$ of white balls and %evolving 
	%Panpan: change evolving to evolves
	evolves over a period of $(t - s)$. By Corollary~\ref{Cor:meanvar}, we have $\E\left[W(t) -a\Delta t  \, \big | \, \field_s\right] = W(s) +  a\Delta (t - s)$. Consequently, we arrive at
	$$\E\left[W(t) -a\Delta t  \, \big | \, \field_s\right] =  W(s) +  a\Delta (t-s) - a\Delta t =  W(s) - a\Delta s,$$
	so $\bigl(W(t) -a\Delta t \bigr)$ is a martingale. Similarly, we have
	$$\E\left[B(t) -a\Delta t  \, \big | \, \field_s\right] =  B(s) - a\Delta s,$$
	At the vectorial level we get
	\begin{align*}
	\E\left[e^{-t \matB}\vecX(t) \Given \field_s\right] 
	&=\begin{pmatrix} W(s) - a\Delta s \\   
	B(s) -   a\Delta s
	\end{pmatrix} \\
	& = \begin{pmatrix} W(s) - a\bigl(B(s) - W(s)\bigr) s \\   
	B(s) -   a\bigl(B(s) - W(s)\bigr) s
	\end{pmatrix}  \\
	& = e^{-s \matB}\vecX(s) ,
	\end{align*}	
	and we conclude that $e^{-t \matB}\vecX(t)$ is a two-dimensional  martingale.
\end{proof}

\begin{remark}
	The mean in Corollary~\ref{Cor:meanvar} can be obtained from the martingale formulation, as we have 
	$$\E\left[e^{-t \matA^\top}\vecX(t) \, \big{|} \, \field_s \right] =e^{-s \matA\top}\vecX(s),$$ 
	for any $s < t$.  We thus have
	$\E\left[e^{-t \matA^\top}\vecX(t)\right] = \vecX(0)$, 
	or in other words, 
	$$\E\left[\vecX(t)\right] =  e^{t \matA^\top}\vecX(0)= \begin{pmatrix} W(0) +  a \Delta t \\
	W(0) + \Delta +a \Delta t  \end {pmatrix}.$$
	However, it is harder to obtain higher moments or the distribution by this technique.
\end{remark}
\begin{lem}
	$\E |W(t)/t -a\Delta |$ is bounded as $t \to \infty$.
\end{lem}	
\begin{proof}
	By the {\em triangle inequality}, we have
	\begin{align*}
	\E\bigl|W(t) - a \Delta t\bigr|  &\le \E \bigl|W(t) - \E[W(t)] \bigr| + \E \bigl|\E[W(t)]- a \Delta t\bigr|
	\\ &= \E \bigl|W(t) - W(0) - a \Delta t \bigr| + W(0),
	\end{align*}
	where the second line is followed by Corollary~\ref{Cor:meanvar}. By {\em Cauchy-Schwarz inequality}, we thus get
	$$
	\E\bigl|W(t) - a \Delta t\bigr|  \le \sqrt{a^2t \bigl(W(0) + \Delta + a \Delta t \bigr)} + W(0),
	$$
	leading to
	$$\E\left|\frac{W(t)}{t} - a \Delta\right| \le \sqrt{a^3 \Delta + \frac{a^2 \bigl(W(0) + \Delta\bigr)}{t}} + \frac{W(0)}{t},$$
	which proves that $\E |W(t)/t -a\Delta |$ is bounded for large value of $t$.
\end{proof}	

\begin{remark}
	The mean, $\E |W(t)/t -a\Delta |$, is not uniformly bounded for all $t \in \mathbb{R}^{+}$; more precisely, it blows up when $t$ is close to $0$. We thus conjecture that $W(t)/t$ may not converge to gamma almost surely.
\end{remark}

%\begin{acknowledgements}
%\end{acknowledgements}

% BibTeX users please use one of
%\bibliographystyle{spbasic}      % basic style, author-year citations
%\bibliographystyle{spmpsci}      % mathematics and physical sciences
%\bibliographystyle{spphys}       % APS-like style for physics
%\bibliography{}   % name your BibTeX data base
% Non-BibTeX users please use

\end{document}